\documentclass{amsart}
\usepackage{amssymb}
\usepackage{amsthm}
\usepackage[utf8]{inputenc}
\usepackage{epsfig}
\usepackage{graphicx}
\usepackage{epstopdf}
\usepackage[usenames]{color}
\usepackage{xcolor}

\newcommand\mut[1]{\ignorespaces}

\newtheorem{thm}{Theorem}
\newtheorem{lem}[thm]{Lemma}
\newtheorem{remark}[thm]{Remark}
\newtheorem{cor}[thm]{Colorally}

\newtheorem{con}[thm]{Conjecture}

\newtheorem{teo}{Teorema}
\newtheorem{lema}[teo]{Lema}

\def\Syl{\mathrm{Syl}}
\def\SL{\mathrm{SL}}
\def\AGL{\mathrm{AGL}}

\def\N{\mathbb{N}}
\def\Z{\mathbb{Z}}

\renewcommand\paragraph[1]{\subsection*{#1}}

\title[Sylow numbers and Mersenne primes]{The Number of Sylow subgroups and a generalization of Mersenne primes}

\author[J. Urroz]{Jorge Urroz}
\address{Departamento de Matem\'aticas e Informática Aplicadas a las Ingenierías Civil y Naval, E.T.S. de Ingeniería de Caminos, Canales y Puertos, Universidad Politécnica de Madrid. Calle Prof. Aranguren 3, 28040 Madrid, Spain.}
\email{jorge.urroz@upm.es}

\author{Alexander Moret\'o}
\address{Departament de Matem\`atiques, Universitat de Val\`encia, 46100
  Burjassot, Val\`encia, Spain}
\email{alexander.moreto@uv.es}

\thanks{The second author was supported by Ministerio de Ciencia e 
Innovaci\'on (Grant  PID2022-137612NB-I00 funded by 
MCIN/AEI/10.13039/501100011033 and ``ERDF A way of making Europe").}

\keywords{Sylow subgroup, Sylow number, prime, repunit}

\subjclass[2020]{Primary 20D20, 11A41, Secondary 11N05,  20D05}

\date{\today}

\begin{document}

\begin{abstract}
Fix an integer $m>2$. We prove that if there exists a finite group with $mp+1$ Sylow $p$-subgroups, where $p$ is large enough, then $mp+1$ is prime. This improves on a theorem of M. Hall and  is a partial answer to Brauer's Problem 26. Our proof uses techniques from analytic number theory, and it also raises new questions in that area. 
\end{abstract}

\maketitle

\section{Introduction}

Problem 26 in Brauer's famous list of open problems \cite{bra} asks for the determination of the possible numbers of Sylow $p$-subgroups of finite groups, where $p$ is a prime. Probably the most significant paper in relation to this problem is M. Hall's 1967 paper \cite{hal}, along with the more recently achieved classification of finite simple groups. By Sylow's theorem, we know that if $\nu_p(G)$ is the number of Sylow $p$-subgroups of a group $G$, then $\nu_p(G)= pm+1$ for some positive integer $m$. When $p=2$, the dihedral group of order $2n$ for $n$ odd has $n$ Sylow $2$-subgroups. Thus the answer to Brauer's problem is clear when $p=2$. It is also easy to see that if $m=1$ and $p$ is any prime, there are groups with $1+p$ Sylow $p$-subgroups. Consider for instance $\SL(2,p)$.

\

Things get more complicated when $p>2$ and $m>1$.  M. Hall proved in Theorem 3.1 of \cite{hal} that if $m>1$, $p>2m+3$ and there is a finite group $G$ with $\nu_p(G)=pm+1$ then $pm+1$ is a prime power. Note that if $pm+1$ is a prime power, $\AGL(1,pm+1)$ has $pm+1$ Sylow $p$-subgroups, so the converse is true. Using this result, he proved in Theorem 3.2 that if $p>5$, then there are no groups with $3p+1$ Sylow $p$-subgroups. 

\

In this note, we aim at answering the following question. Are there other values of $m$ such that if $p$ is large enough, then there are no groups with $pm+1$ Sylow $p$-subgroups?  Using  the standard notation 
$\varphi(n)$ to be the Euler totient function  counting the number of positive integers smaller than $n$ and coprime to $n$, we prove the following 

\begin{thm}
\label{1}
Let $m>2$ be an integer such that $m+1$ is not a prime number, and let $p$ be a prime.  If $p>(m+1)^{\varphi(m)}/m$ and  $G$ is a finite group with $pm+1$ Sylow $p$-subgroups then $mp+1$ is prime and $G$ is $p$-solvable.
\end{thm}

Note that for odd $m>1$,  and $p>2$, $m+1$ and $mp+1$ are never prime, so this extends Hall's Theorem 3.2 to any odd $m$.

The key new tool in the proof of Theorem \ref{1} is the following result from number theory that has independent interest.

\begin{thm}
\label{2}
Given a positive  integer $m$, there exist finitely many positive integers $k,p,q$ with $k\ge 2$, $p$ prime and  $q\ne m+1$, such that 
$$
pm+1=q^k.
$$
Concretely we get $p\le C_m$ for $C_m=(m+1)^{\varphi(m)}/m$.

\end{thm}

It seems likely that the hypothesis $q\ne m+1$ is necessary in Theorem \ref{2}, even if we assume that $q$ is prime. This would mean that the hypothesis $m+1$ is not prime is necessary in Theorem 1. Proving this, however, appears to be a very complicated problem, related to the existence of infinitely many (generalized) Mersenne primes. We refer the reader to  Section \ref{mer} for more details on this. These are interesting questions in number theory, that will be studied more deeply elsewhere. 
  
The layout of this paper is as follows. We prove Theorem 2 and study related questions in number theory in Section 2. We use this to complete the proof of Theorem 1 in Section 4. In Section 3 we discuss the open questions on what we call generalized Mersenne numbers.   

\section{A diophantine equation}

We are interested in the solutions in integers of 
\begin{equation}\label{eq:eq}
pm+1=q^k,
\end{equation}
for $p$ a prime number and $m,q,k$ positive integers. We are particularly interested in the existence of solutions with $q$ also prime.

\subsection{Fixing $p$}

\begin{thm} For any prime $p$  and integer $k\ge 1$ there are infinitely many integers $m$ and primes $q$\ such that 
$$
pm+1=q^k.
$$
\end{thm}
\begin{proof} By Dirichlet Theorem there are infinitely many primes $q\equiv 1\pmod p$, so $p|q^k-1$ and we can consider the solution given by $q$ and  $m=(q^k-1)/p$.
\end{proof}

\subsection{Fixing $m$}

For the case $k=1$ given any $m$ and  prime $p$,  we have the trivial solution $q=mp+1\in \N$. Also, if we fix an even integer $m$, it is conjectured that $pm+1=q$ has infinitely many solutions $p,q$ prime numbers as a consequence of Dickson's conjeture applied to the linear equations  $n$ and $nm+1$, since precisely the conjecture claims that both will be prime for infinitely many $n$. Finally note that for $m$ odd $pm+1$ is even, unless $p=2$, so there is at most one prime solution $q=2m+1$.

\

Now consider $k\ge 2$ and $m$ fixed. The following is Theorem \ref{2}.
\begin{thm} Given a positive  integer $m$, there exist finitely many positive integers $k,p,q$ with $k\ge 2$, $p$ prime and  $q\ne m+1$, such that 
$$
pm+1=q^k.
$$
Concretely we get $p\le C_m$ for $C_m=(m+1)^{\varphi(m)}/m$.
\end{thm}

\begin{proof} Suppose $p,q,k$  is a solution to Equation \ref{eq:eq} with $k\ge 2$.  Note that if $q>p$, then $pm+1\ge(p+1)^2$, which is impossible if $m\le 2$ and in general it gives $p\le m-2\le C_m$, since $m^2-2m\le (m+1)^2\le (m+1)^{\varphi(m)}$. Note that $\varphi(m)\ge 2$ for any $m\ge 3$. Hence, from now on  we  assume $q<p$, and  in particular  $p\nmid q-1$ . If $l|k$ and $l\ne k$ then we know that 
\begin{equation}\label{eq:divir}
q^k-1=(q^l-1)M_l,
\end{equation} 
for some integer $M_l\ge q^l+1$, since $q^k-1\ge q^{2l}-1=(q^l-1)(q^l+1)$. Using identity  (\ref{eq:divir}) for $l=1$ in equation (\ref{eq:eq}),  we get that $p|M_1$ since $p\nmid q-1$, and hence, $(q-1)|m$, 
which, from our hypothesis $q\ne m+1$, gives $q\le \frac m2+1$. 

\

It is clear that $(q,m)=1$ so let $a=$ord$_m(q)$, the multiplicative order of $q$ in $(\Z/m\Z)^*$. If $a=1$, then  $q=ms+1$, which can not happen since $q\le \frac m2+1$, so $a\ge 2$ and, since $a|k$,  we have a solution to

$$
pm+1=q^{an},
$$
for some integer $n\ge 1$.  Suppose $n=1$.  Then
$$
pm=q^a-1< \left(\frac m2+1\right)^{\varphi(m)},
$$
which gives again $p\le C_m$.

\

Assume then that $n\ge 2$. This time we use  (\ref{eq:divir}) with $l=a$. Then, by definition of $a$ we get that $m|q^a-1$, and hence,  $m=q^a-1$  since otherwise $p$ cannot be prime. 

\medskip

We now  use (\ref{eq:divir}) with $l=n$. 
 If $n> a\ge 2$ then,  $M_n\ge q^n+1>q^n-1> q^a-1=m$. Then if $d=(m,q^n-1)$ we have 
$$
p=\left(\frac{q^n-1}{d}\right)\left(\frac{M_n}{m/d}\right),
$$
and $p$ can not be prime since both factors are bigger than $1$,  so it follows that  $n\le a$. Hence
$$
pm< \left(q^{a}\right)^a= (m+1)^{a}\le (m+1)^{\varphi(m)}.
$$
which gives again $p<C_m$ and the result follows.

\end{proof}

From the proof we get the following interesting corollary. For any integer $n$, we will use $\tau(n)$ to denote the number of divisors of $n$, as usual.
\begin{cor} Given a positive  integer $m$, such that $m+1$ is not a perfect power,  there exist at most $m-2+\tau(m)$ triples  of integers $(k,p,q)$ with $k\ge 2$, $p$ prime  and integer $q\ne m+1$, such that 
$$
pm+1=q^k.
$$
\end{cor}
\begin{proof}  Following the proof of the previous theorem we see that if $p<q$, then $p\le m-2$. But then,  $q^2\le q^k=pm+1\le (m-1)^2$, so $q\le m-1$. On the other hand  if $q<p$ then we get $q\le \frac m2+1\le m$. So, in any case ord$_m(q)=a>1$ and since $m+1$ is not a perfect power then $pm+1=q^a$ is deduced in the theorem. Then if we let $(m,q-1)=d$, then 
$$
p=\frac{q-1}{d}\frac{\sum_{j=0}^{a-1}q^j}{m/d}
$$ 
so either $q-1=d$, or $\sum_{j=0}^{a-1}q^j=\frac md$. In the first case $q-1$ is a divisor of $m$, $a$ the order of $d+1$ modulo $m$ and $p=\frac{q^a-1}m$. So at most there are $\tau(m)$ triples of this type. In the second case $p=\frac{q-1}{d}<q$ so $p\le m-2$ and hence $q\le m-1$ as seen before. So,  noting that $q>1$, we take each $2\le q\le m-1$ then $a$ is fixed to be ord$_m(q)$ and once $a,q$ are fixed again $p=\frac{q^a-1}m$, so there are at most $m-2$ triples of this type. Adding both cases we obtain the result.

\end{proof}
It is interesting to note also  that, in fact, for $p>m$, $k=ord_p(q)=b$ since, otherwise, setting  $b=$ord$_p(q)$ we get 
$$
pm=q^{bl}-1\ge q^{2b}-1=(q^b-1)(q^b+1)>p^2,
$$
which gives $p<m$. In particular it follows that any prime $p$ such that  $pm+1=q^k$ with $q\ne m+1$ verifies that $ord_p(q)<m^2$.

\section{Generalized Mersenne numbers}\label{mer}

As we have mentioned, the hypothesis $m+1$ not prime seems necessary. Indeed, Equation \ref{eq:eq} could have infinitely many solutions of the form $pm+1=(m+1)^k$ for $p=((m+1)^k-1)/m$ a prime number. For example if $m=1$ these would be Mersenne primes, and we expect that there are infinitely many of them. In general, primes of the form 
$p=(m^k-1)/(m-1)$ for some integer  $m$ are called repunits in base $m$, and according to the generalized repunit conjecture, if for a given $m$ there is a repunit prime in base $m$, then there are infinitely many of them. We call those $m$ \lq\lq suitable\rq\rq and it is clear that not every $m$ is suitable, for example if  $m$ is a perfect power, the sequence $(m^k-1)/(m-1)$ does not  have any single prime. For more examples  of non suitable integers,  we address the interested reader to check the literature related with aurelian factorizations, see for example \cite{and}. The name repunit comes from the fact that, once represented in base $m$ it contains only the digit $1$ in the representation.  A simple program using python already shows the existence of repunit primes.  We include on Table 1  a few examples  for $m=3, 5, 6,7,10,11,12, 13,14,15$ and $17$.
\begin{table}
\caption{Repunit primes}
$$
\begin{array}{|l|c|c|}\hline
m&k&p\\
\hline
3&3&13\\
&7&1093\\
&13&797161\\
\hline
5&3&31\\
&7&19531\\
&11&12207031\\
&13&305175781\\
&47&177635683940025046467781066894531\\
\hline
6&2&7\\
&3&43\\
&7&55987\\
&29&7369130657357778596659\\
\hline
7&5&2801\\
&13&16148168401\\
\hline
10&2&11\\
&19&1111111111111111111\\
&23&11111111111111111111111\\
\hline
11&17&50544702849929377\\
&19&6115909044841454629\\
\hline
12&2&13\\
&3&157\\
&5&22621\\
&19&29043636306420266077\\
\hline
13&5& 30941\\
&7&5229043\\
\hline
14&7& 8108731\\
&19&459715689149916492091\\
\hline
15&3&241\\
\hline
17&3&307\\
&5&88741\\
&7&25646167\\
&11&2141993519227\\
\hline
\end{array}
$$
\end{table}

\

Note that in  all those solutions $k$ must have prime since, otherwise, if $k=ab$ then $(m^k-1)/(m-1)$ has the trivial factor $(m^a-1)/(m-1)$.

\

The existence of repunit primes is one of the most important problems in analytic number theory, and has a vast literature around the subject. Even to prove the problem letting the base vary is really hard. Of course one can consider repunit primes of the form $11$ by just considering the base $m=p-1$. However, just trying to prove the existance of  infinitely many repunit primes with three digits would be as proving the existence of infinitely many primes of the form $m^2+m+1$, another of the most classical problems about  distribution of primes.

\

In this line  we include  a result proving the existence of infinitely many repunit primes, with a fixed number of digits, which is a direct consequence of Bunyakovsky conjecture. We first  include it  for convenience to the reader

\begin{con}[Bunyakovsky]\label{con:bu} Let $f(x)=\sum_{j=0}^{d}a_jx^j\in\Z[x]$ be an irreducible  polynomial with positive  leading coefficient , and 
such that gcd$(a_1,\dots, a_d)=1$. Then
$f(n)$ is prime for infinitely many positive integers  $n$.
\end{con}

\

\begin{thm} For any given  prime  number $l$, if Bunyankovsky conjectyure is true, there exist infinitely many repunit primes with $l$ digits.  
\end{thm}
\begin{proof} For $l=2$, the result is trivially true unconditionally by just selecting the base $m=p-1$. For $l\ge 3$, prime,  we just have to notice that the polynomial $f(x)=\frac{x^l-1}{x-1}$ is irreducible, so we can apply Conjecture \ref{con:bu}, and clearly if $f(m)=p$, prime, then $p$ is a  repunit with $l$ digits  on base $m$.
\end{proof}

\begin{remark} We would like to finish this section with a very simple remark emphasizing this surprising conection between finite groups and the existence of repunit primes. In particular finding groups with $mp+1$ $p$-Sylow groups will be useful to find repunit primes, namely  if $pm+1=(m+1)^k$, then $p=\sum_{j=0}^{k-1}(m+1)^j$ would be a repunit in base $m+1$.

\end{remark}

\section{Sylow numbers}

Let $m\geq2$ be an integer and let $p>2$ be  prime. If $mp+1$ is prime, we say that $mp+1$ is an $m$-Sophie Germain prime. In particular, $2$-Sophie Germain primes are the usual Sophie Germain primes.  Note that if $mp+1$ is an $m$-Sophie Germain prime then $m$ is even. By Dickson's conjecture, it is expected that for every even $m$, there are infinitely many $m$- 
Sophie German primes. 

\

We will use  the following result, which is a consequence of Theorem 2.1 of \cite{hal}

\begin{lem}
\label{lemhal}
Let $p$ be a prime number and let $G$ be a finite group. If $S$ is a composition factor (or a chief factor) of $G$, then $\nu_p(S)$ divides $\nu_p(G)$.
\end{lem}

The following is Theorem \ref{1}.  

\begin{thm}
Let $m>2$ be an integer such that $m+1$ is not a prime number, and let $p$ be a prime.  If $p>(m+1)^{\varphi(m)}/m$ and  $G$ is a finite group with $mp+1$ Sylow $p$-subgroups then $mp+1$ is prime and $G$ is $p$-solvable.
\end{thm}

\begin{proof}
By Theorem 3.1 of \cite{hal}, we know that $mp+1=q^k$ for some prime $q$ and integer $k\geq1$. By Theorem 2, $k=1$. 

Now, suppose that $G$ is not $p$-solvable. By Lemma \ref{lemhal}, we know that if $K/L$ is a chief factor of $G$, then $\nu_p(K/L)$ divides $\nu_p(G)$. Since $\nu_p(G)=q$ is prime, if follows that if $K/L$ is a nonabelian chief factor of $G$ of order divisible by $p$, then $K/L=S$ is simple and $\nu_p(S)=q=mp+1$. 

Arguing as in the proof of Theorem 3.1 of \cite{hal}, we can conclude that this is not possible. For the convenience of the reader, we sketch the nice argument.

Let $P\in\Syl_p(S)$. 
Since $m<p$, $S$ is isomorphic to a subgroup of a symmetric group on less than $p^2$ letters. Hence $P$ is elementary abelian. Suppose that $|P|>p$ and let $Q$ be another Sylow $p$-subgroup of $S$. Then $|P:P\cap Q|<mp+1<p^2$, so $|P:P\cap Q|=p$. Using Brodkey's theorem (Theorem 1.37 of \cite{isa}), we conclude that $|P|=p$. 

Now, we can use Brauer's theory of  groups with a Sylow $p$-subgroup of order $p$. By Theorem 2 of \cite{br},  
\begin{equation}\label{eq:easy}
m=\frac{hup+u^2+u+h}{u+1}
\end{equation} 
for some positive  integers $u$ and $h$. Using the hypothesis that $p>(m+1)^{\varphi(m)}/m$ it is easy to see that this is not possible.  Indeed,  $m=1,2,4,6$ are the only positive integers such that $\varphi(m)\le 2$ but, by hypothesis $m>2$, and $4,6$ can be checked one by one. For the rest of integers $m$, we have $\varphi(m)\ge 3$ and  $(m+1)^{\varphi(m)}/m>m^2+3m$, which, from the identity (\ref{eq:easy}) gives, noting that $\frac u{u+1}\ge \frac12$ for any positive integer $u$,
$$
m>\frac{m^2+3m}{2},
$$
which is impossible.
\end{proof}

As we have mentioned, for  any fixed even $m$, we expect $mp+1$ to be prime infinitely many times.  The bound $p>(m+1)^{\varphi(m)}/m$ surely is not best possible. 
However, we cannot assume in general that $p>2m+3$ as in Hall's theorem. As an example, we include   Table 2. These are all the primes $1000>p>2m+3$ such that $pm+1$ is a prime power that is not prime and it is not a repunit, already included in Table 1. Observe that for $m=7$ and $m=26$, $m+1=q^3$, and in the rest we see that ord$_{23}(2)=11$, ord$_{54}(19)=3$, ord$_{90}(31)=3$, and ord$_{126}(43)=3$. Note that indeed, $24,55,91,127$ are not perfect powers. 

\begin{table}
\caption{Primes $1000>p>2m+3$, $k\ge2$, $q\ne m+1$}
$$
\begin{array}{|l|c|c|c|}\hline
m&q&k&p\\
\hline
7&2&9&73\\
23&2&11&89\\
26&3&9&757\\
54&19&3&127\\
90&31&3&331\\
126&43&3&631\\
\hline
\end{array}
$$
\end{table}

\

Using the classification of finite simple groups, we can get $p$-solvability in M. Hall's situation.

\begin{thm}
Let $m>1$ be an  integer and let $p>2$ be a prime.  If $p>2m+3$ and  $G$ is a finite group with $pm+1$ Sylow $p$-subgroups then $G$ is $p$-solvable.
\end{thm}

\begin{proof}
By way of contradiction,  suppose that $G$ is not $p$-solvable. Let $S$ be a non-abelian simple composition factor of $G$ of order divisible by $p$, so that by Lemma \ref{lemhal}, $\nu_p(S)$ is a divisor of $pm+1$. In particular, $\nu_p(S)<p^2$. We obtain a contradiction from Theorem 3.1 of \cite{gls}.  
\end{proof}

{\bf Acknowledgment.} We would like to thank the referee for a careful reading of the paper. His comments have improved the final presentation.

\end{document}